\documentclass[12pt]{amsart}
\usepackage{accents,amsfonts,amsmath,amssymb,amsthm,comment,hyperref,mathrsfs,mathtools,multirow,times}

\usepackage[top=1.1in,bottom=1.1in,left=1.06in,right=1.06in]{geometry}

\newtheorem{theorem}{Theorem}[section]
\newtheorem{lemma}[theorem]{Lemma}

\theoremstyle{definition}
\newtheorem{definition}[theorem]{Definition}

\theoremstyle{remark}
\numberwithin{equation}{section}

\allowdisplaybreaks

\newcommand\mylabel[1]{\label{#1}}

\newcommand\thm[1]{\ref{thm:#1}}
\newcommand\lem[1]{\ref{lem:#1}}

\newcommand\eqn[1]{(\ref{eq:#1})}

\newcommand\subsect[1]{\ref{subsec:#1}}

\newcommand*{\bigdot}[1]{%
\accentset{\mbox{\large\bfseries .}}{#1}}

\begin{document}

\title[Sum of parts in overpartitions and partitions without repeated odd parts]{Sum of parts in overpartitions and partitions without repeated odd parts}

\author{Frank Garvan}
\address{\newline \newline Department of Mathematics, University of Florida, Gainesville, FL 32608, USA \newline}
\email{fgarvan@ufl.edu}
\author{Rishabh Sarma}
\address{\newline \newline Department of Mathematics, The Pennsylvania State University, McAllister Building, University Park, PA 16802, USA \newline}
\email{rishabh.sarma@psu.edu}
\maketitle

\begin{abstract}
In this paper, we obtain several Ramanujan-type congruences modulo $5$ and $7$ for sum of certain non-overlined parts in overpartitions classified by parity and sum of certain parts in partitions without repeated odd parts classified by parity. Our proofs for the congruences are elementary, depending only on classical theta function identities.
\end{abstract}

\section{Introduction and results}

The partition function $p(n)$, defined as the number of unrestricted partitions of a nonnegative integer $n$, occupies a central place in combinatorics and number theory. One of the most remarkable features of this function is the existence of the celebrated congruences discovered by Ramanujan.
$$p(5n+4)\equiv 0 \pmod{5},$$
$$p(7n+5)\equiv 0 \pmod{7},$$
$$p(11n+6)\equiv 0 \pmod{11}.$$
\\
Over the years, the search for such congruences has expanded far beyond the classical partition function. Authors have since found Ramanujan-type congruences for functions arising from distinct partitions, $t$-core partitions, colored partitions, overpartitions, partition pairs among many others and numerous statistics such as ranks, cranks and smallest parts. 
\\\\
Very recently, Andrews and Dastidar \cite{An-Da26} considered the function $\text{SOME}(n)$, defined as the sum of all the odd parts minus the sum of all the even parts in all the partitions of $n$, found and treated its generating function to arrive at results of the following type, arguably most notably the $5n+4$ congruence satisfied by $\text{SOME}(n)$.

\begin{theorem}
\mylabel{thm:An-Da1}{\cite[Theorem 4]{An-Da26}}
$$\mathrm{SOME}(5n + 4) \equiv 0 \pmod{5}.$$
\end{theorem}

\begin{theorem}
\mylabel{thm:An-Da2}{\cite[Theorem 5]{An-Da26}}
$$\mathrm{SOME}(5n + 2) \equiv 0 \pmod{5}.$$
\end{theorem}

\begin{theorem}
\mylabel{thm:An-Da3}{\cite[Corollary 6]{An-Da26}}
The sum of all the odd parts in the partitions of $5n + 4$ is divisible by $5$.
\end{theorem}

\begin{theorem}
\mylabel{thm:An-Da4}{\cite[Corollary 7]{An-Da26}}
The sum of all the even parts in the partitions of $5n + 4$ is divisible by $5$.
\end{theorem}

In this paper, we study along similar lines for overpartitions and partitions without repeated odd parts to obtain Ramanujan-type congruences via elementary methods. 

\begin{definition}
Let $\text{SENO}(n)$ and $\text{SONO}(n)$ denote the sum of all the even non-overlined parts and the sum of all the odd non-overlined parts in all the overpartitions of $n$ respectively.
\end{definition}

\begin{definition}
Let $\text{SNO}(n)$ denote the sum of all the non-overlined parts in all the overpartitions of $n$.
\end{definition}

Then we have the following. 

\begin{theorem}
\mylabel{thm:SNOcong}
$$\mathrm{SNO}(5n + 4) \equiv 0 \pmod{5},$$
$$\mathrm{SNO}(5n + 2) \equiv 0 \pmod{5},$$
$$\mathrm{SNO}(7n + 3) \equiv 0 \pmod{7}.$$
In other words, the sum of all the non-overlined parts in all the overpartitions of $5n+4$ and $5n+2$ is divisible by $5$ and that in all the overpartitions of $7n+3$ is divisible by $7$.
\end{theorem}

\begin{theorem}
\mylabel{thm:SENOcong}
$$\mathrm{SENO}(7n + 5) \equiv 0 \pmod{7}.$$
In other words, the sum of all the even non-overlined parts in all the overpartitions of $7n+5$ is divisible by $7$.
\end{theorem}

\begin{theorem}
\mylabel{thm:SONOcong}
$$\mathrm{SONO}(5n + 3) \equiv 0 \pmod{5}.$$
In other words, the sum of all the odd non-overlined parts in all the overpartitions of $5n+3$ is divisible by $5$.
\end{theorem}

\begin{definition}
Let $\text{SEDO}(n)$ and $\text{SODO}(n)$ respectively denote the sum of all the even parts and the sum of all the odd parts in all the partitions of $n$ without repeated odd parts.
\end{definition}

Then we have the following.

\begin{theorem}
\mylabel{thm:SEDOcong}
$$\mathrm{SEDO}(5n) \equiv 0 \pmod{5},$$
$$\mathrm{SEDO}(7n) \equiv 0 \pmod{7}.$$
In other words, the sum of all the even parts in all the partitions of $5n$ without repeated odd parts is divisible by $5$ and that in all the partitions of $7n$ without repeated odd parts is divisible by $7$.
\end{theorem}

\begin{theorem}
\mylabel{thm:SODOcong}
$$\mathrm{SODO}(5n) \equiv 0 \pmod{5},$$
$$\mathrm{SODO}(7n) \equiv 0 \pmod{7},$$
$$\mathrm{SODO}(7n+2) \equiv 0 \pmod{7}.$$
In other words, the sum of all the odd parts in all the partitions of $5n$ without repeated odd parts is divisible by $5$ and that in all the partitions of $7n$ and $7n+2$ without repeated odd parts is divisible by $7$. 
\end{theorem}

\begin{theorem}
\mylabel{thm:SEDO-SODOcong}
$$\mathrm{SEDO}(7n+5)-\mathrm{SODO}(7n+5) \equiv 0 \pmod{7}.$$
In other words, the sum of all the even parts minus the sum of all the odd parts in all the partitions of $7n+5$ without repeated odd parts is divisible by $7$.
\end{theorem}

The rest of the paper is organized as follows. In the next section, we collect some preliminary results. Section 3 is devoted to finding dissections of some objects that we will use to prove our main results in the subsequent sections. We prove the congruences for overpartitions in Section 4 and those for partitions without repeated odd parts in Section 5.

\section{Preliminaries}
Let $N,m,n$ be non-negative integers. Throughout we use the following standard notations.
\begin{align*}
(a)_N = (a;q)_N &:= \prod_{k=0}^{N-1}(1-aq^k),
\\
(a)_{\infty} = (a;q)_{\infty} &:= \lim_{N \mapsto \infty}(a)_N\,\,\text{where}\,\,\lvert q\rvert<1.
\end{align*}
\\
We define the products
$$J_{b,a}(q) = J_{b,a} := (q^a; q^b)_\infty (q^{b-a};q^b)_\infty (q^b;q^b)_\infty,$$
$$J_{b,0}(q):=(q^b;q^b)_\infty.$$
This satisfies 
\begin{equation}
\label{eq:Jrule}
J_{b,a+b}(q)=-q^{-a}J_{b,a}(q).
\end{equation}
Throughout the rest of the paper, the base in $J_{b,a}$ is meant to be $q$ when we do not mention it explicitly.
\\\\
For a formal power series
$$F(q) = \sum_{n=0}^\infty f(n) q^n,$$
we define the operator
\begin{equation}
\label{eq:Updef}
U_{p,a}(F) := \sum_{n=0}^\infty f(pn +a) q^n.
\end{equation}
If
$$F(q)=\sum_{r=0}^{p-1}q^rF_r(q^p)$$
is a \(p\)-dissection of $F(q)$, where $p$ is a prime, then
$$q\frac{d}{dq}F(q)=\sum_{r=0}^{p-1}q^r\left(rF_r(q^p)+pq^pF_r'(q^p)\right).$$
Hence,
\begin{equation}
\label{eq:Fdissder}
q\frac{d}{dq}F(q)\equiv \sum_{r=0}^{p-1}r q^rF_r(q^p)\pmod p.\end{equation}
Euler's pentagonal number theorem, Jacobi's identity, Jacobi's triple product identity and Watson's quintuple product identity respectively are 
\begin{equation}
\label{eq:PNT}
E(q):=(q;q)_\infty=\sum_{n=-\infty}^{\infty} (-1)^n q^{\frac{n(3n-1)}{2}},
\end{equation}
\begin{equation}
\label{eq:Jacobi}
E_3(q):=(q;q)_\infty^3 = \sum_{n=0}^{\infty} (-1)^n (2n+1) q^{\frac{n(n+1)}{2}},
\end{equation}
\begin{equation}
\label{eq:JTP}
\sum\limits_{n=-\infty}^{\infty}(-1)^nz^nq^{\frac{n(n-1)}{2}}=(zq;q)_\infty\left(z^{-1} q;q\right)_\infty (q;q)_\infty,
\end{equation}
and
\begin{equation}
\label{eq:QTP}
\sum_{n=-\infty}^{\infty}\left(z^{3n}-z^{-3n-1}\right)q^{\frac{n(3n+1)}{2}} =(q;q)_\infty (z;q)_\infty (z^{-1}q;q)_\infty(z^2q;q^2)_\infty (z^{-2}q;q^2)_\infty.
\end{equation}
Next we note the well known $5$ (see \cite[Lemma 3.18]{Ga88}) and $7$-dissections (see \cite[Section 3.6]{An-Be12} or \cite[Section 5]{Ga88}) of $E(q)$
\begin{equation}
\label{eq:Ediss5}
E(q)=\frac{J_{25,0}J_{25,10}}{J_{25,5}}-qJ_{25,0}-q^2\frac{J_{25,0}J_{25,5}}{J_{25,10}},
\end{equation}
\begin{equation}
\label{eq:Ediss7}
E(q) = \frac{J_{49,14} J_{49,0}}{J_{49,7}}-\frac{q J_{49,21} J_{49,0}}{J_{49,14}}-q^{2} J_{49,0}+\frac{q^{5} J_{49,7} J_{49,0}}{J_{49,21}},
\end{equation}
and the following identities that follow from the quintuple product identity (see \cite[Equations 7 and 8]{On-Ro95})
\begin{equation}
\label{eq:etaprodsum1}
\frac{(q;q)_\infty^5}{(q^2;q^2)_\infty^2}=\sum_{n=-\infty}^{\infty}(6n+1)\,q^{\frac{3n^{2}+n}{2}},
\end{equation}
\begin{equation}
\label{eq:etaprodsum2}
\frac{(q^2;q^2)_\infty^5}{(q;q)_\infty^2}=\sum_{n=-\infty}^{\infty}
(-1)^n(3n+1)q^{3n^2+2n}.
\end{equation}
We also need Weierstrass' identity below for products of four theta functions.
\begin{theorem}
\mylabel{thm:Weier}{\cite[Theorem 4.1]{Br-Th22}}
For any non-zero complex numbers $a,b,c$ and $d$,
$$\theta(d)\,\theta\left(\frac{b}{c}\right)\,\theta(abc)\,\theta(ad)=\theta(b)\,\theta\left(\frac{d}{c}\right)\,\theta(acd)\,\theta(ab)
-\frac{b}{c}\,\theta(c)\,\theta\left(\frac{d}{b}\right)\,\theta(abd)\,\theta(ac),$$
where the function $\theta$ is defined by
$$\theta(z):=\theta(z;q)=(z;q)_\infty \left(\frac{q}{z};q\right)_\infty (q;q)_\infty.$$
\end{theorem}
\noindent
\\
Finally to aid our calculations in the proofs below, let us define 
$$B(q)=\sum_{n=1}^{\infty}\frac{q^n}{(1-q^n)^2},$$
and 
$$\mathrm{DO}(n)=\frac{(q;q)_\infty}{(q^2;q^2)_\infty^2}.$$
Then, it is easy to deduce that 
\begin{equation}
\label{eq:Eder}
\bigdot{E}(q):=q\frac{d}{dq}E(q)=-E(q)B(q),
\end{equation}
\begin{equation}
\label{eq:E3der}
\bigdot{E}_3(q):=q\frac{d}{dq}E_3(q)=-3E_3(q)B(q),
\end{equation}
and 
\begin{equation}
\label{eq:DOder}
\bigdot{\mathrm{DO}}(q):=q\frac{d}{dq}\mathrm{DO}(q)=\mathrm{DO}(q)(-B(q)+4B(q^2)).
\end{equation}

\section{Some dissections modulo $5$ and $7$}
In this section, we state and prove two lemmas on dissections of some objects that we will use in the proof of our results in the next two sections.

\begin{lemma}
\label{lem:E3diss57}
\begin{align*}
E_3(q)&\equiv J_{25,10}+2qJ_{25,5} \pmod 5,
\\
E_3(q)&\equiv J_{7,3}(q^7)-3qJ_{7,2}(q^7)-2q^3J_{7,1}(q^7) \pmod 7.
\end{align*}
\end{lemma}

\begin{proof}
Using Jacobi's identity in Equation \eqn{Jacobi}, we have
\begin{align*}
E_3(q)&=\sum_{n=0}^{\infty}(-1)^n(2n+1)q^{\frac{n(n+1)}{2}}
\\
&=\sum_{n=0}^{\infty}\sum_{j=0}^{4}(-1)^{5n+j}\big(2(5n+j)+1\big)q^{\frac{(5n+j)(5n+j+1)}{2}}
\\
&\equiv \sum_{n=0}^{\infty}(-1)^n\left(q^{\frac{5n(5n+1)}{2}}+2q^{\frac{(5n+1)(5n+2)}{2}}-2q^{\frac{(5n+3)(5n+4)}{2}}-q^{\frac{(5n+4)(5n+5)}{2}}\right)\pmod 5
\\
&=\sum_{n=0}^{\infty}(-1)^n\left(q^{\frac{5n(5n+1)}{2}}-q^{\frac{(5n+4)(5n+5)}{2}}\right)
\\
&\quad+2\sum_{n=0}^{\infty}(-1)^n\left(q^{\frac{(5n+1)(5n+2)}{2}}-q^{\frac{(5n+3)(5n+4)}{2}}\right)
\\
&=\sum_{n=-\infty}^{\infty}(-1)^nq^{\frac{5n(5n+1)}{2}}+2\sum_{n=-\infty}^{\infty}(-1)^nq^{\frac{(5n+1)(5n+2)}{2}}
\\
&=(q^{10};q^{25})_{\infty}(q^{15};q^{25})_{\infty}(q^{25};q^{25})_{\infty}+2q\,(q^{20};q^{25})_{\infty}(q^{5};q^{25})_{\infty}(q^{25};q^{25})_{\infty}
\\
&=J_{25,10}+2qJ_{25,5}.
\end{align*}
where in the last but one step, we have used Jacobi's triple product identity in Equation \eqn{JTP} with $q\mapsto q^{25}$, and $z=q^{10}$ and $z=q^{5}$ for the two bilateral sums respectively.
\\\\
Next, starting from the same identity, we have
\begin{align*}
E_3(q)&=\sum_{n=0}^{\infty}(-1)^n(2n+1)q^{\frac{n(n+1)}{2}}
\\
&=\sum_{n=0}^{\infty}\sum_{j=0}^{6}(-1)^{7n+j}\big(2(7n+j)+1\big)q^{\frac{(7n+j)(7n+j+1)}{2}}
\\
&\equiv \sum_{n=0}^{\infty}(-1)^n\Big(q^{\frac{7n(7n+1)}{2}}-3q^{\frac{(7n+1)(7n+2)}{2}}+5q^{\frac{(7n+2)(7n+3)}{2}}
\\
&\quad+2q^{\frac{(7n+4)(7n+5)}{2}}+3q^{\frac{(7n+5)(7n+6)}{2}}-q^{\frac{(7n+6)(7n+7)}{2}}\Big)\pmod 7
\\
&=\sum_{n=0}^{\infty}(-1)^n\left(q^{\frac{7n(7n+1)}{2}}-q^{\frac{(7n+6)(7n+7)}{2}}\right)-3\sum_{n=0}^{\infty}(-1)^n\left(q^{\frac{(7n+1)(7n+2)}{2}}-q^{\frac{(7n+5)(7n+6)}{2}}\right)
\\
&\quad+5\sum_{n=0}^{\infty}(-1)^n\left(q^{\frac{(7n+2)(7n+3)}{2}}-q^{\frac{(7n+4)(7n+5)}{2}}\right)\pmod 7
\\
&=\sum_{n=-\infty}^{\infty}(-1)^nq^{\frac{7n(7n+1)}{2}}-3\sum_{n=-\infty}^{\infty}(-1)^nq^{\frac{(7n+1)(7n+2)}{2}}-2\sum_{n=-\infty}^{\infty}(-1)^nq^{\frac{(7n+2)(7n+3)}{2}}
\\
&=(q^{21};q^{49})_{\infty}(q^{28};q^{49})_{\infty}(q^{49};q^{49})_{\infty}-3q\,(q^{14};q^{49})_{\infty}(q^{35};q^{49})_{\infty}(q^{49};q^{49})_{\infty}
\\
&\quad-2q^3\,(q^{7};q^{49})_{\infty}(q^{42};q^{49})_{\infty}(q^{49};q^{49})_{\infty}
\\
&=J_{7,3}(q^7)-3qJ_{7,2}(q^7)-2q^3J_{7,1}(q^7),
\end{align*}
where in the last but one step, we have used Jacobi's triple product identity stated in Equation \eqn{JTP} with $q\mapsto q^{49}$, and $z=q^{28},z=q^{35}$ and $z=q^{28}$ for the three bilateral sums respectively.
\end{proof}

\begin{lemma}
\label{lem:etaprodsumdiss7}
\begin{align*}
\frac{(q;q)_{\infty}^5}{(q^2;q^2)_{\infty}^2}&\equiv\frac{\sqrt{J_{14,7}(q^7)}}{J_{14,0}(q^7)^{5/2}}\Big(J_{14,1}(q^7)J_{14,5}(q^7)J_{14,6}(q^7)+2qJ_{14,2}(q^7)J_{14,3}(q^7)J_{14,5}(q^7)
\\
&\quad+3q^5J_{14,1}(q^7)J_{14,3}(q^7)J_{14,4}(q^7)\Big) \pmod 7,
\\
\frac{(q^2;q^2)_{\infty}^5}{(q;q)_{\infty}^2}&\equiv\frac{J_{14,6}(q^7)J_{14,0}(q^7)}{J_{14,3}(q^7)}+2q\frac{J_{14,2}(q^7)J_{14,0}(q^7)}{J_{14,1}(q^7)}+3q^5\frac{J_{14,4}(q^7)J_{14,0}(q^7)}{J_{14,5}(q^7)} \pmod 7.
\end{align*}
\end{lemma}

\begin{proof}
Using the identity in Equation \eqn{etaprodsum1}, we have
\begin{align*}
\frac{(q;q)_{\infty}^5}{(q^2;q^2)_{\infty}^2}&=\sum_{n=-\infty}^{\infty}(6n+1)q^{\frac{3n^2+n}{2}}
\\
&=\sum_{n=-\infty}^{\infty}\sum_{j=0}^{6}(6(7n+j)+1)q^{\frac{3(7n+j)^2+(7n+j)}{2}}
\\
&\equiv \sum_{n=-\infty}^{\infty}\Big(q^{\frac{3(7n)^2+7n}{2}}-q^{\frac{3(7n+2)^2+(7n+2)}{2}}+2q^{\frac{3(7n+6)^2+(7n+6)}{2}}-2q^{\frac{3(7n+3)^2+(7n+3)}{2}}
\\
&\quad+3q^{\frac{3(7n+5)^2+(7n+5)}{2}}-3q^{\frac{3(7n+4)^2+(7n+4)}{2}}\Big)\pmod 7
\\
&\equiv \sum_{n=-\infty}^{\infty}\Big(q^{\frac{7(21n^2+n)}{2}}-q^{\frac{7(21n^2+13n+2)}{2}}\Big)+2q\sum_{n=-\infty}^{\infty}\Big(q^{\frac{7(21n^2+37n+16)}{2}}-q^{\frac{7(21n^2+19n+4)}{2}}\Big)
\\
&\quad+3q^5\sum_{n=-\infty}^{\infty}\Big(q^{\frac{7(21n^2+31n+10)}{2}}-q^{\frac{7(21n^2+25n+6)}{2}}\Big)
\\
&=(q^{49},q^7,q^{42};q^{49})_{\infty}(q^{63},q^{35};q^{98})_{\infty}+2q(q^{49},q^{14},q^{35};q^{49})_{\infty}(q^{77},q^{21};q^{98})_{\infty}
\\
&\quad+3q^5(q^{49},q^{21},q^{28};q^{49})_{\infty}(q^{91},q^7;q^{98})_{\infty}
\\
&=\frac{\sqrt{J_{14,7}(q^7)}}{J_{14,0}(q^7)^{5/2}}\Big(J_{14,1}(q^7)J_{14,5}(q^7)J_{14,6}(q^7)+2qJ_{14,2}(q^7)J_{14,3}(q^7)J_{14,5}(q^7)
\\
&\quad+3q^5J_{14,1}(q^7)J_{14,3}(q^7)J_{14,4}(q^7)\Big),
\end{align*}
where to obtain the last but one step, we have used quintuple product identity stated in Equation \eqn{QTP} with $q\mapsto q^{49}$, and $z=q^{7},z=q^{14}$ and $z=q^{21}$ for the three bilateral sums respectively.
\\\\
Next, using the identity in Equation \eqn{etaprodsum2}, we have
\begin{align*}
\frac{(q^2;q^2)_{\infty}^5}{(q;q)_{\infty}^2}&=\sum_{n=-\infty}^{\infty}(-1)^n(3n+1)q^{3n^2+2n}
\\
&=\sum_{n=-\infty}^{\infty}\sum_{j=0}^{6}(-1)^{7n+j}\left(3(7n+j)+1\right)q^{3(7n+j)^2+2(7n+j)}
\\
&\equiv\sum_{n=-\infty}^{\infty}(-1)^n\Big(q^{147n^2+14n}+3q^{147n^2+56n+5}+4q^{147n^2+140n+33}
\\
&\quad+6q^{147n^2+182n+56}+5q^{147n^2+224n+85}+5q^{147n^2+266n+120}\Big)\pmod 7
\\
&\equiv\sum_{n=-\infty}^{\infty}(-1)^n q^{7(21n^2+2n)}+q^{21}\sum_{n=-\infty}^{\infty}(-1)^n q^{7(21n^2-16n)}
\\
&\quad+2q\sum_{n=-\infty}^{\infty}(-1)^n q^{7(21n^2-4n)}+2q^8\sum_{n=-\infty}^{\infty}(-1)^n q^{7(21n^2-10n)}
\\
&\quad+3q^5\sum_{n=-\infty}^{\infty}(-1)^n q^{7(21n^2+8n)}-3q^{33}\sum_{n=-\infty}^{\infty}(-1)^n q^{7(21n^2-20n)}\pmod 7
\\
&=J_{42,19}(q^7)+q^{21}J_{42,5}(q^7)+2q\Big(J_{42,17}(q^7)+q^7J_{42,11}(q^7)\Big)
\\
&\quad+3q^5\Big(J_{42,13}(q^7)-q^{28}J_{42,1}(q^7)\Big)
\\
&=\frac{J_{14,6}(q^7)J_{14,0}(q^7)}{J_{14,3}(q^7)}+2q\frac{J_{14,2}(q^7)J_{14,0}(q^7)}{J_{14,1}(q^7)}+3q^5\frac{J_{14,4}(q^7)J_{14,0}(q^7)}{J_{14,5}(q^7)},
\end{align*}
where to obtain the last but one step, we have used Jacobi's triple product identity stated in Equation \eqn{JTP}. The final line follows from the identities
$$J_{42,19}(q)+q^3J_{42,5}(q)=\frac{J_{14,6}(q)J_{14,0}(q)}{J_{14,3}(q)},$$
$$J_{42,17}(q)+qJ_{42,11}(q)=\frac{J_{14,2}(q)J_{14,0}(q)}{J_{14,1}(q)},$$
and
$$J_{42,13}(q)-q^4J_{42,1}(q)=\frac{J_{14,4}(q)J_{14,0}(q)}{J_{14,5}(q)},$$
with $q\mapsto q^7$. 
\\\\
These three identities are special cases of the quintuple product identity which in one of its forms, equivalent to \eqn{QTP}, due to Hickerson \cite[Theorem 1.0]{Hi88} is
$$\theta(qx^3;q^3)+x\theta(q^2x^3;q^3)=\frac{(q;q)_\infty \theta(x^2;q)}{\theta(x;q)}.$$
Replacing $q\mapsto q^{14}$ and $x=q^a$ gives
\begin{equation}
\label{eq:J14}
\frac{J_{14,0}J_{14,2a}}{J_{14,a}}=J_{42,14+3a}+q^aJ_{42,28+3a}.
\end{equation}
Substituting $a=3$ in Equation \eqn{J14} gives
$$\frac{J_{14,0}J_{14,6}}{J_{14,3}}=J_{42,23}+q^3J_{42,37}=J_{42,19}+q^3J_{42,5}.$$
Substituting $a=1$ in Equation \eqn{J14} gives
$$\frac{J_{14,0}J_{14,2}}{J_{14,1}}=J_{42,17}+qJ_{42,31}=J_{42,17}+qJ_{42,11}.$$
Substituting $a=5$ in Equation \eqn{J14} gives
$$\frac{J_{14,0}J_{14,10}}{J_{14,5}}=J_{42,29}+q^5J_{42,43}.$$
Using \eqn{Jrule}, we have the third identity
$$\frac{J_{14,0}J_{14,4}}{J_{14,5}}=J_{42,13}-q^4J_{42,1}.$$
\end{proof}

\section{Proof of the congruences for overpartitions}

\subsection{Proof of Theorem \thm{SNOcong}}
\label{subsec:SNOcongproof}
We have
\begin{align*}
\sum_{n=0}^{\infty} \mathrm{SNO}(n) q^n&=\frac{\partial}{\partial z}\Bigg|_{z=1}(-q;q)_{\infty}\prod_{n=1}^{\infty} \frac{1}{(1 - (zq)^{n})} 
\\
&=\frac{(-q;q)_{\infty}}{(q;q)_{\infty}}\,\sum_{n=1}^{\infty}\frac{n q^{n}}{1 - q^{n}} 
\\
&=\frac{(-q;q)_{\infty}}{(q;q)_{\infty}}\,\sum_{m=1}^{\infty}\sum_{n=1}^{\infty}nq^{mn}
\\
&=\frac{(-q;q)_{\infty}}{(q;q)_{\infty}}\,\sum_{m=1}^{\infty}\frac{q^{m}}{(1-q^{m})^2}
\\
&=\frac{(q^2;q^2)_{\infty}}{(q;q)_{\infty}^2}\,B(q)
\\
&=\frac{(q^2;q^2)_{\infty}}{(q;q)_{\infty}^5}\,E_3(q)B(q).
\end{align*}
Using Equation \eqn{Fdissder} and the $5$-dissection of $E_3(q)$ in Lemma \lem{E3diss57}, we find the $5$-dissection of $\bigdot{E}_3(q)$ to get
$$\bigdot{E}_3(q)\equiv 2qJ_{25,5}.$$
Therefore using Equation \eqn{E3der} we have
$$-3E_3(q)B(q)\equiv 2qJ_{25,5}\pmod 5,$$
and hence
\begin{equation}
\label{eq:E3Bdiss5}
E_3(q)B(q)\equiv qJ_{25,5}\pmod 5.
\end{equation}
Using Euler's pentagonal number theorem in Equation \eqn{PNT} and the above congruence, we get
$$\sum_{n=0}^{\infty}\mathrm{SNO}(n)q^n\equiv\frac{\displaystyle\sum_{n=-\infty}^{\infty}(-1)^n q^{n(3n-1)}}{(q^5;q^5)_{\infty}}\,qJ_{25,5}\pmod 5.$$
Since
$$n(3n-1)+1=3n^2-n+1$$
is not of the form $5m+2$ or $5m+4$ for any $n$, we get
$$\mathrm{SNO}(5n+2)\equiv 0\pmod 5,$$
and
$$\mathrm{SNO}(5n+4)\equiv 0\pmod 5.$$
To prove the mod $7$ congruence, we use Equation \eqn{Fdissder} and the $7$-dissection of $E_3(q)$ in Lemma \lem{E3diss57} to find the $7$-dissection of $\bigdot{E}_3(q)$ and get
$$\bigdot{E}_3(q)\equiv 4qJ_{49,14}+q^3J_{49,7}\pmod 7.$$
Using Equation \eqn{E3der} we have
$$-3E_3(q)B(q)\equiv 4qJ_{49,14}+q^3J_{49,7}\pmod 7,$$
and hence
\begin{align*}
E_3(q)B(q)\equiv 8qJ_{49,14}+2q^3J_{49,7}\equiv qJ_{49,14}+2q^3J_{49,7}\pmod 7.
\end{align*}
Since 
$$\sum_{n=0}^{\infty} \mathrm{SNO}(n) q^n=\frac{(q^2;q^2)_{\infty}}{(q;q)_{\infty}^5}\,E_3(q)B(q) \equiv \frac{E(q)^2 E(q^2) E_3(q) B(q)}{(q^7;q^7)_\infty}  \pmod{7},$$ 
proving the congruence $\mathrm{SNO}(7n + 3) \equiv 0 \pmod{7}$ is equivalent to showing that
$$U_{7,3}\left( E(q)^2 E(q^2) E_3(q) B(q) \right) 
\equiv 0 \pmod{7},$$
where the $U_{p,a}$ operator is defined in Equation \eqn{Updef}. Substituting for the 7-dissection of $E(q)$ from Equation \eqn{Ediss7} and for $E_3(q) B(q)$ from the congruence above, this in turn is equivalent to showing that
\begin{align*}
C(q) &= 2 q^{2} J_{7,1}^{5} J_{7,2} J_{14,2} J_{14,4}+4 q J_{7,1}^{4} J_{7,2} J_{7,3} J_{14,4}^{2}+2 q J_{7,1}^{3} J_{7,2}^{2} J_{7,3} J_{14,2} J_{14,6}
\\
&+4 q J_{7,1}^{3} J_{7,3}^{3} J_{14,2} J_{14,4}+2 q J_{7,1}^{2} J_{7,2}^{3} J_{7,3} J_{14,2} J_{14,4}-J_{7,1}^{2} J_{7,3}^{4} J_{14,4}^{2}+J_{7,2}^{4} J_{7,3}^{2} J_{14,2} J_{14,6}
\\
&\equiv 0 \pmod{7}.
\end{align*}
We prove this using the two theta function identities
\begin{align*}
C_1(q) &:= q J_{7,1}^{3} J_{7,2}+J_{7,1} J_{7,3}^{3}-J_{7,2}^{3} J_{7,3} = 0,~\mbox{and}
\\
C_2(q) &:= q J_{7,1}^{2} J_{14,2} J_{14,4}+J_{7,1} J_{7,3} J_{14,4}^{2}-J_{7,2} J_{7,3} J_{14,2} J_{14,6} =0.
\end{align*}
The first identity follows from substituting $a=q,b=q,c=q^3$, and $d=q^2$ and taking the base $q^{14}$ in the function $\theta(z;q)$ in Weierstrass's identity in Theorem \thm{Weier}. The second identity $C_2(q)=0$ on splitting modulus $7$ to modulus $14$ using 
$$J_{7,1}=J_{14,1}J_{14,6},\quad J_{7,2}=J_{14,2}J_{14,5},\quad J_{7,3}=J_{14,3}J_{14,4},$$
equivalently reduces to showing 
$$qJ_{14,1}^{2}J_{14,2}J_{14,6}+J_{14,1}J_{14,3}J_{14,4}^{2}-J_{14,2}^{2}J_{14,3}J_{14,5}=0.$$
This once again follows from substituting $a=q,b=q,c=q^3$, and $d=q^2$ and taking the base $q^{14}$ in the function $\theta(z;q)$ in the Weierstrass's identity.
\\\\
The result 
$$C(q)  \equiv 0 \pmod{7}$$
follows from the the fact that
$$C(q) - \left( Q_1(q)\, C_1(q) + Q_2(q)\,C_2(q)\right)= -7 J_{7,1}^{2} J_{7,3}^{4} J_{14,4}^{2}+7 J_{7,2}^{4} J_{7,3}^{2} J_{14,2} J_{14,6},$$
where
\begin{align*}
Q_1(q) &= 2 J_{7,1} J_{7,3} J_{14,4}^{2}+4 J_{7,2} J_{7,3} J_{14,2} J_{14,6},~\mbox{and}
\\
Q_2(q) &= 2 q J_{7,1}^{3} J_{7,2}+4 J_{7,1} J_{7,3}^{3}+2 J_{7,2}^{3} J_{7,3}.
\end{align*}  
This proves the congruence $$\mathrm{SNO}(7n + 3) \equiv 0 \pmod{7}.$$

\subsection{Proof of Theorem \thm{SENOcong}}
We have
\begin{align*}
\sum_{n=0}^{\infty} \mathrm{SENO}(n) q^n&=\frac{\partial}{\partial z}\Bigg|_{z=1}\frac{(-q;q)_{\infty}}{(q;q^2)_{\infty}}\prod_{n=1}^{\infty} \frac{1}{(1 - (zq)^{2n})} 
\\
&=\frac{(-q;q)_{\infty}}{(q;q^2)_{\infty}}\,\frac{1}{(q^2;q^2)_{\infty}}\sum_{n=1}^{\infty} \frac{2n q^{2n}}{1 - q^{2n}} 
\\
&=\frac{(-q;q)_{\infty}}{(q;q)_{\infty}}\,\sum_{m=1}^{\infty}\frac{2q^{2m}}{(1-q^{2m})^2}
\\
&=\frac{(q^2;q^2)_{\infty}}{(q;q)_{\infty}^2}\,2B(q^2)
\\
&=2\frac{(q;q)_{\infty}^5}{(q^2;q^2)_{\infty}^2(q;q)_{\infty}^7}\,E_3(q^2)B(q^2).
\end{align*}
In the proof of the mod $7$ congruence in the previous theorem, we had obtained
\begin{align*}
E_3(q)B(q)\equiv qJ_{49,14}+2q^3J_{49,7}\pmod 7
\end{align*}
which implies
\begin{align*}
E_3(q^2)B(q^2)\equiv q^2J_{98,28}+2q^6J_{98,14}\pmod 7.
\end{align*}
Using the identity in Equation \eqn{etaprodsum1} and the above congruence, we get
$$\sum_{n=0}^{\infty}\mathrm{SENO}(n)q^n\equiv\left(2\sum_{n=-\infty}^{\infty}(6n+1)q^{\frac{3n^2+n}{2}}\right)\left(q^2J_{98,28}+2q^6J_{98,14}\right)\pmod 7.$$
Since
$$\frac{3n^2+n}{2}+2\quad\text{and}\quad\frac{3n^2+n}{2}+6$$
are not of the form $7m+5$ for any $n$, we get
$$\mathrm{SENO}(7n+5)\equiv 0\pmod 7.$$

\subsection{Proof of Theorem \thm{SONOcong}}
We have
\begin{align*}
\sum_{n=0}^{\infty} \mathrm{SONO}(n) q^n&=\sum_{n=0}^{\infty} \mathrm{SNO}(n) q^n-\sum_{n=0}^{\infty} \mathrm{SENO}(n) q^n
\\
&=\frac{(-q;q)_{\infty}}{(q;q)_{\infty}}\,\left(\sum_{m=1}^{\infty}\frac{q^{m}}{(1-q^{m})^2}-\sum_{m=1}^{\infty}\frac{2q^{2m}}{(1-q^{2m})^2}\right)
\\
&=\frac{(q^2;q^2)_{\infty}}{(q;q)_{\infty}^2}\left(B(q)-2B(q^2)\right)
\\
&=\frac{(q^2;q^2)_{\infty}}{(q;q)_{\infty}^5}E_3(q)B(q)-2\frac{E_3(q)}{(q;q)_{\infty}^5}E(q^2)B(q^2).
\end{align*}
Then using Equations \eqn{Ediss5} and \eqn{E3Bdiss5}, we have
\begin{align*}
\frac{(q^2;q^2)_{\infty}}{(q;q)_{\infty}^5}E_3(q)B(q)&\equiv \frac{\frac{J_{50,0}J_{50,20}}{J_{50,10}}-q^2J_{50,0}-q^4\frac{J_{50,0}J_{50,10}}{J_{50,20}}}{(q^5;q^5)_{\infty}}qJ_{25,5} \pmod 5
\\
&=\frac{q\frac{J_{25,5}J_{50,0}J_{50,20}}{J_{50,10}}-q^3J_{25,5}J_{50,0}-q^5\frac{J_{25,5}J_{50,0}J_{50,10}}{J_{50,20}}}{(q^5;q^5)_{\infty}}.
\end{align*} 
We treat the other term next. Using Equation \eqn{Fdissder} and the $5$-dissection of $E(q)$ in Equation \eqn{Ediss5}, we have
\begin{align*}
\bigdot{E}(q)&\equiv -qJ_{25,0}-2q^2\frac{J_{25,0}J_{25,5}}{J_{25,10}} \pmod 5.
\end{align*}
Therefore using Equation \eqn{Eder} we have
$$E(q)B(q)\equiv qJ_{25,0}+2q^2\frac{J_{25,0}J_{25,5}}{J_{25,10}} \pmod 5.$$
Using the $5$-dissection of $E_3(q)$ in Lemma \lem{E3diss57} and the above congruence, we get
\begin{align*}
2\frac{E_3(q)}{(q;q)_{\infty}^5}E(q^2)B(q^2) &\equiv 2\frac{J_{25,10}+2qJ_{25,5}}{(q^5;q^5)_{\infty}}(q^2J_{50,0}+2q^4\frac{J_{50,0}J_{50,10}}{J_{50,20}}) \pmod 5
\\
&=\frac{2q^2J_{25,10}J_{50,0}+4q^3J_{25,5}J_{50,0}+4q^4\frac{J_{25,10}J_{50,0}J_{50,10}}{J_{50,20}}-2q^5\frac{J_{25,5}J_{50,0}J_{50,10}}{J_{50,20}}}{(q^5;q^5)_{\infty}}.
\end{align*}
Thus in the generating function of $\mathrm{SONO}(n)$ that we had obtained as a difference of the two terms that we have treated, it is easy that the two terms of form $q^{5n+3}$ cancel out modulo $5$. This gives the desired congruence
$$\mathrm{SONO}(5n + 3) \equiv 0 \pmod{5}.$$

\section{Proof of the congruences for partitions without repeated odd parts}

\subsection{Proof of Theorem \thm{SEDOcong}}
\label{subsec:SEDOcongproof}
We have
\begin{align*}
\sum_{n=0}^{\infty} \mathrm{SEDO}(n) q^n&=\frac{\partial}{\partial z}\Bigg|_{z=1}(-q;q^2)_{\infty}\prod_{n=1}^{\infty} \frac{1}{(1 - (zq)^{2n})} 
\\
&=\frac{(-q;q^2)_{\infty}}{(q^2;q^2)_{\infty}}\,\sum_{n=1}^{\infty}\frac{2n q^{2n}}{1 - q^{2n}} 
\\
&=2\frac{(-q;q^2)_{\infty}}{(q^2;q^2)_{\infty}}\,\sum_{m=1}^{\infty}\frac{q^{2m}}{(1-q^{2m})^2}
\\
&=2\frac{(-q;q^2)_{\infty}}{(q^2;q^2)_{\infty}^5}\,(q^2;q^2)_{\infty}^4\,B(q^2)
\\
&=2\frac{(-q;q^2)_{\infty}(q^2;q^2)_{\infty}}{(q^2;q^2)_{\infty}^5}\,E_3(q^2)B(q^2).
\end{align*}
For simplification, we let $q\mapsto -q$ and consider the function $$2\frac{(q;q^2)_{\infty}(q^2;q^2)_{\infty}}{(q^2;q^2)_{\infty}^5}\,E_3(q^2)B(q^2)=2\frac{(q;q)_{\infty}}{(q^2;q^2)_{\infty}^5}\,E_3(q^2)B(q^2).$$
Previously in Equation \eqn{E3Bdiss5} we had found
$$E_3(q)B(q)\equiv qJ_{25,5}\pmod 5.$$
Then using Equation \eqn{PNT} and the above congruence, we get
$$\sum_{n=0}^{\infty}\mathrm{SEDO}(n)q^n\equiv\frac{2\displaystyle\sum_{n=-\infty}^{\infty}(-1)^n q^{\frac{n(3n-1)}{2}}}{(q^{10};q^{10})_{\infty}}\,q^2J_{50,10}\pmod 5.$$
Since
$$\frac{n(3n-1)}{2}+2$$
is not of the form $5m$ for any $n$, we get
$$\mathrm{SEDO}(5n)\equiv 0\pmod 5.$$
To prove the modulo $7$ congruence, we appeal to Equation \eqn{DOder}  which states
\begin{equation}
\label{eq:DOeq}
\bigdot{\mathrm{DO}}(q)=-\mathrm{DO}(q)B(q)+4\mathrm{DO}(q)B(q^2).
\end{equation}
We just found that the generating function of $\mathrm{SEDO}(n)$ is 
$$2\frac{(-q;q^2)_{\infty}}{(q^2;q^2)_{\infty}}\,B(q^2).$$
For simplification, we let $q\mapsto -q$ and consider the function 
$$2\frac{(q;q^2)_{\infty}}{(q^2;q^2)_{\infty}}\,B(q^2)$$
which is $2\mathrm{DO}(q)B(q^2)$. We then prove the desired congruence modulo $7$ by treating the other two terms in Equation \eqn{DOeq}.
\\\\
We have 
$$\mathrm{DO}(q)B(q)=\frac{(q;q)_{\infty}^5}{(q^2;q^2)_{\infty}^2(q;q)_{\infty}^7}\,E_3(q)B(q).$$
In the proof of the mod $7$ congruence in Theorem \thm{SNOcong} in Section \subsect{SNOcongproof}, we had obtained
\begin{align*}
E_3(q)B(q)\equiv qJ_{7,2}(q^7)+2q^3J_{7,1}(q^7)\pmod 7.
\end{align*}
Then using the first congruence in Lemma \lem{etaprodsumdiss7} and the above congruence we get
\begin{align*}
\mathrm{DO}(q)B(q)&\equiv
\frac{\sqrt{J_{14,7}(q^7)}}{J_{14,0}(q^7)^{5/2}J_{1,0}(q^7)}\Big(qJ_{7,2}(q^7)J_{14,1}(q^7)J_{14,5}(q^7)J_{14,6}(q^7)
\\
&\quad+2q^2J_{7,2}(q^7)J_{14,2}(q^7)J_{14,3}(q^7)J_{14,5}(q^7)
\\
&\quad+2q^3J_{7,1}(q^7)J_{14,1}(q^7)J_{14,5}(q^7)J_{14,6}(q^7)
\\
&\quad+4q^4J_{7,1}(q^7)J_{14,2}(q^7)J_{14,3}(q^7)J_{14,5}(q^7)
\\
&\quad+3q^6J_{7,2}(q^7)J_{14,1}(q^7)J_{14,3}(q^7)J_{14,4}(q^7)
\\
&\quad-q^8J_{7,1}(q^7)J_{14,1}(q^7)J_{14,3}(q^7)J_{14,4}(q^7)\Big)\pmod 7.
\end{align*}
Thus, $\mathrm{DO}(q)B(q)$ does not have terms of the form $q^{7n}$ or $q^{7n+5}$.
\\\\
Next we have
\begin{equation}
\label{eq:DOmult}
\mathrm{DO}(q)=\frac{(q^2;q^2)_{\infty}^5}{(q;q)_{\infty}^2(q^2;q^2)_{\infty}^7}\,E_3(q).
\end{equation}
Multiplying the $7$-dissections of $\frac{(q^2;q^2)_{\infty}^5}{(q;q)_{\infty}^2}$ and $E_3(q)$ in Lemmas \lem{etaprodsumdiss7} and \lem{E3diss57} respectively we get
\begin{align*}
\mathrm{DO}(q)&\equiv\frac{1}{J_{2,0}(q^7)}\Bigg(\frac{J_{7,3}(q^7)J_{14,6}(q^7)J_{14,0}(q^7)}{J_{14,3}(q^7)}+2q\frac{J_{7,3}(q^7)J_{14,2}(q^7)J_{14,0}(q^7)}{J_{14,1}(q^7)}
\\
&\quad+3q^5\frac{J_{7,3}(q^7)J_{14,4}(q^7)J_{14,0}(q^7)}{J_{14,5}(q^7)}-3q\frac{J_{7,2}(q^7)J_{14,6}(q^7)J_{14,0}(q^7)}{J_{14,3}(q^7)}
\\
&\quad+q^2\frac{J_{7,2}(q^7)J_{14,2}(q^7)J_{14,0}(q^7)}{J_{14,1}(q^7)}-2q^6\frac{J_{7,2}(q^7)J_{14,4}(q^7)J_{14,0}(q^7)}{J_{14,5}(q^7)}
\\
&\quad-2q^3\frac{J_{7,1}(q^7)J_{14,6}(q^7)J_{14,0}(q^7)}{J_{14,3}(q^7)}-4q^4\frac{J_{7,1}(q^7)J_{14,2}(q^7)J_{14,0}(q^7)}{J_{14,1}(q^7)}
\\
&\quad+q^8\frac{J_{7,1}(q^7)J_{14,4}(q^7)J_{14,0}(q^7)}{J_{14,5}(q^7)}\Bigg)\pmod 7.
\end{align*}
Then using Equation \eqn{Fdissder}, we have
\begin{align*}
\bigdot{\mathrm{DO}}(q)&\equiv \frac{1}{J_{2,0}(q^7)}\Bigg(2q\frac{J_{7,3}(q^7)J_{14,2}(q^7)J_{14,0}(q^7)}{J_{14,1}(q^7)}+q^5\frac{J_{7,3}(q^7)J_{14,4}(q^7)J_{14,0}(q^7)}{J_{14,5}(q^7)}
\\
&\quad-3q\frac{J_{7,2}(q^7)J_{14,6}(q^7)J_{14,0}(q^7)}{J_{14,3}(q^7)}+2q^2\frac{J_{7,2}(q^7)J_{14,2}(q^7)J_{14,0}(q^7)}{J_{14,1}(q^7)}
\\
&\quad+2q^6\frac{J_{7,2}(q^7)J_{14,4}(q^7)J_{14,0}(q^7)}{J_{14,5}(q^7)}+q^3\frac{J_{7,1}(q^7)J_{14,6}(q^7)J_{14,0}(q^7)}{J_{14,3}(q^7)}
\\
&\quad-2q^4\frac{J_{7,1}(q^7)J_{14,2}(q^7)J_{14,0}(q^7)}{J_{14,1}(q^7)}+q^8\frac{J_{7,1}(q^7)J_{14,4}(q^7)J_{14,0}(q^7)}{J_{14,5}(q^7)}\Bigg)\pmod 7.
\end{align*}
Thus, $\bigdot{\mathrm{DO}}(q)$ as well does not have terms of the form $q^{7n}$. Equation \eqn{DOeq} then implies the congruence $$\mathrm{SEDO}(7n)\equiv 0\pmod 7.$$

\subsection{Proof of Theorem \thm{SODOcong}}
$$\mathrm{SODO}(5n)\equiv 0\pmod 5$$
follows from 
$$\mathrm{SEDO}(5n)\equiv 0\pmod 5$$ 
because $\mathrm{SODO}(n)+\mathrm{SEDO}(n)$ gives $np^*(n)$ where $p^*(n)$ is the number of partitions of $n$ without repeated odd parts. The congruence $$\mathrm{SODO}(7n)\equiv 0\pmod 7$$ follows by similar reasoning.
\\\\
Next we have 
\begin{align*}
\sum_{n=0}^\infty(\mathrm{SODO}(n)+\mathrm{SEDO}(n))q^n&=\sum_{n=0}^\infty np^*(n)q^n=q\frac{d}{dq}\sum_{n=0}^\infty p^*(n) q^n=q\frac{d}{dq}\frac{(-q;q^2)_{\infty}}{(q^2;q^2)_{\infty}}.
\end{align*}
Then
\begin{align*}
\sum_{n\geq 0}\left(\mathrm{SEDO}(n)-\mathrm{SODO}(n)\right)q^n&=2\sum_{n\geq 0}\mathrm{SEDO}(n)q^n-\sum_{n=0}^\infty(\mathrm{SODO}(n)+\mathrm{SEDO}(n))q^n
\\
&=4\frac{(-q;q^2)_\infty}{(q^2;q^2)_\infty}B(q^2)-q\frac{d}{dq}\frac{(-q;q^2)_\infty}{(q^2;q^2)_\infty} 
\\
&=\frac{(-q;q^2)_\infty}{(q^2;q^2)_\infty}B(-q),
\end{align*}
where to obtain the last step we have used Equation \eqn{DOder} with $q\mapsto -q$. Then,
$$2\sum_{n\geq 0}\mathrm{SODO}(n)q^n=q\frac{d}{dq}\frac{(-q;q^2)_{\infty}}{(q^2;q^2)_{\infty}}-\frac{(-q;q^2)_\infty}{(q^2;q^2)_\infty}B(-q).$$
Replacing $q\mapsto -q$ again, we get
$$2\sum_{n\geq 0}(-1)^n\mathrm{SODO}(n)q^n=\bigdot{\mathrm{DO}}(q)-\mathrm{DO}(q)B(q).$$
From the proof of the previous theorem, we find that 
\begin{align*}
U_{7,2}(\bigdot{\mathrm{DO}}(q)-\mathrm{DO}(q)B(q))&\equiv 2q^2\frac{J_{7,2}(q^7)J_{14,2}(q^7)J_{14,0}(q^7)}{J_{14,1}(q^7)J_{2,0}(q^7)}
\\
&\quad-\frac{\sqrt{J_{14,7}(q^7)}}{J_{14,0}(q^7)^{5/2}J_{1,0}(q^7)}2q^2J_{7,2}(q^7)J_{14,2}(q^7)J_{14,3}(q^7)J_{14,5}(q^7) 
\\
&\quad\pmod 7
\\
&=0,
\end{align*}
where to obtain the last equality we have used 
$$J_{2,0}J_{14,0}^2=J_{14,2}J_{14,4}J_{14,6}~~~~\text{and}~~~~
J_{1,0}J_{14,0}^{11/2}=\sqrt{J_{14,7}}\,J_{14,1}J_{14,2}J_{14,3}J_{14,4}J_{14,5}J_{14,6}.$$
This implies the congruence 
$$\mathrm{SODO}(7n+2) \equiv 0 \pmod{7}.$$

\subsection{Proof of Theorem \thm{SEDO-SODOcong}}

In the proof of the previous theorem, we had obtained
\begin{align*}
\sum_{n\geq 0}\left(\mathrm{SEDO}(n)-\mathrm{SODO}(n)\right)q^n=-\frac{(-q;q^2)_\infty}{(q^2;q^2)_\infty}B(-q),
\end{align*}
Then
$$\sum_{n\geq 0}(-1)^n\left(\mathrm{SEDO}(n)-\mathrm{SODO}(n)\right)q^n=-\frac{(q;q^2)_\infty}{(q^2;q^2)_\infty}B(q)=\mathrm{DO}(q)B(q).$$
In the proof of Theorem \thm{SEDOcong} in Section \subsect{SEDOcongproof}, we had found that $\mathrm{DO}(q)B(q)$ does not have terms of the form $q^{7n+5}$. This gives us the desired congruence 
$$\mathrm{SEDO}(7n+5)-\mathrm{SODO}(7n+5) \equiv 0 \pmod{7}.$$


\begin{thebibliography}{00}

\bibitem{An-Be12}
G.~E. Andrews and B.~C. Berndt, {\it Ramanujan's lost notebook. Part III}, Springer, New York, 2012.

\bibitem{An-Da26}
G.~E. Andrews and M.~G. Dastidar, $p(5n+4)$ again, Ramanujan J. {\bf 69} (2026), no.~1, Paper No. 26, 9 pp.

\bibitem{Br-Th22}
J.~G. Bradley-Thrush, Properties of the Appell-Lerch function (I), Ramanujan J. {\bf 57} (2022), no.~1, 291--367.

\bibitem{Ga88}
F.~G. Garvan, New combinatorial interpretations of Ramanujan's partition congruences mod $5,7$ and $11$, Trans. Amer. Math. Soc. {\bf 305} (1988), no.~1, 47--77.

\bibitem{Hi88}
D.~R. Hickerson, A proof of the mock theta conjectures, Invent. Math. {\bf 94} (1988), no.~3, 639--660.

\bibitem{On-Ro95}
K. Ono and S. Robins, Superlacunary cusp forms, Proc. Amer. Math. Soc. {\bf 123} (1995), no.~4, 1021--1029.

\end{thebibliography}
\end{document}